\documentclass{amsart}
\usepackage{pstricks,pstricks-add,pst-node}
\usepackage{tikz}
\usepackage{url}

\newtheorem{Theorem}{Theorem}[section]

\newtheorem{Cor}[Theorem]{Corollary}

\newcommand{\tor}{\stackrel{r}{\rightarrow}}
\newcommand{\torr}{\stackrel{R}{\rightarrow}}

\begin{document}
\date{April 21, 2011}

\title{Rainbow induced subgraphs in proper vertex colorings}
\author{Andrzej Kisielewicz and Marek Szyku\l a}

\address{University of Wroc\l aw\\
Department of Mathematics and Computer Science\\
pl. Grunwaldzki 2, 50-384 Wroc\l aw, Poland}
\email{andrzej.kisielewicz@math.uni.wroc.pl, marek.szykula@ii.uni.wroc.pl}
\thanks{Supported in part by Polish MNiSZW grant N N201 543038.}
\keywords{rainbow, induced subgraph, vertex coloring, replication graph}


\begin{abstract}
For a given graph $H$ we define $\rho(H)$ to be the minimum order of a graph $G$ such that every proper vertex coloring of $G$ contains a rainbow induced subgraph isomorphic to $H$. We give upper and lower bounds for $\rho(H)$, compute the exact value for some classes of graphs, and consider an interesting combinatorial problem connected with computation of $\rho(H)$ for paths. This research is motivated by some ideas in on-line graph coloring algorithms.  
\end{abstract}

\maketitle

Rainbow induced subgraphs have been considered in many papers in connection with various problems of extremal graph theory. They have been considered both for edge-colorings and vertex-colorings, and both in terms of existence or in terms of avoiding (see \cite{AC,AI,AM,AS,KMSV,KT,KL}). Our special motivation comes from research in on-line coloring (see \cite{BCP,Ma}), where the base for some algorithms is the existence of rainbow anticliques to force a player to use a new color. In particular, in \cite{Ma}, a problem has been formulated to estimate the minimal number of moves in the game considered one needs to force the appearance of a rainbow copy of a fixed graph $H$ in a fixed class of graphs $C$. 

In this paper we deal with proper vertex colorings of graphs, \emph{colorings}, in short, and \emph{rainbow induced subgraphs}, by which we mean induced subgraphs whose all vertices have different colors. We consider a problem of constructing a minimal graph $G$ such that for a given graph $H$ every coloring of $G$ contains a rainbow induced subgraph isomorphic to $H$.  We shall write simply $G\tor H$ to denote such a situation. We define $\rho(H)$ to be the least number $m$ such that there exist a graph $G$ with $G \tor H$.  

A similar definitions are introduced in \cite{AM} in connection with anti-Ramsey numbers. There is no requirement for colorings to be proper. Instead, there is a restriction on the number of colors used to be not smaller than the cardinality $|H|$. As a consequence it may happen there are only finitely many graphs that in every such coloring contain rainbow colored copies of $H$, and the authors concentrate on such situations. Proper vertex colorings and a problem connected with the existence of rainbow induced stars and paths have been considered in \cite{KT}. Interesting results on proper coloring for edges are contained in \cite{KMSV} and \cite{KL}.

\section{The bounds}

It is not difficult to see that for every graph $H$ there exists a graph $G$ satisfying $G\tor H$. In our first result we establish a relatively tight bounds for $\rho(H)$.

\begin{Theorem}\label{th1}
Let $m'=m'(H)$ denotes the number of non-edges in a graph $H$, $\chi = \chi(H)$, and $n = |H|$. Then the following holds
\begin{equation}
\left\lceil \frac{n}{\chi} \right\rceil \left( n - \frac{\chi}{2}\left\lceil\frac{n}{\chi}-1\right\rceil\right)  \leq \rho(H) \leq n+m' \label{eq1}
\end{equation} 
\end{Theorem}

\begin{proof}
In order to prove the second inequality we construct a suitable graph $G$ with $G\tor H$.

Let $h_1,h_2,...,h_n$ denote the vertices of $H$ in a fixed order. We define $G$ as consisting of $n$ disjoint cliques  $KH_1,KH_2,...,KH_n$ corresponding to the vertices of $H$ and connected in a way reflecting the structure of $H$. The starting clique $KH_1$ has precisely one vertex. 
Let us suppose we have defined the subgraph $G_{i-1}$ of $G$ consisting of the cliques  $KH_1,KH_2,...,KH_{i-1}$, $1<i\leq n$ and some edges among them, and let $m_i$ denotes the number of edges in $H$ of the form $(h_j,h_i)$ for $1\leq j < i$. We define now $G_{i}$ as the union of $G_{i-1}$ and the clique $KH_{i}$ consisting of $i - m_i$ vertices, and some additional edges between $G_{i-1}$ and $KH_{i}$ according to the following rule: if for some $1\leq j < i$, there is an edge $h_jh_{i}$ in $H$, then we add all the possible edges between vertices of cliques $KH_j$ and $KH_i$; otherwise, no edge between the cliques $KH_j$ and $KH_i$ is added.
(In Figure~\ref{fig1} an example is shown for a graph on 5 vertices).

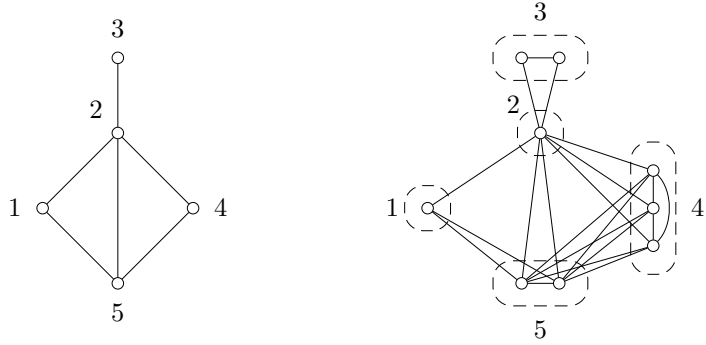
\begin{figure*}\begin{center}

\begin{pspicture}(3,3)(-1,-1.5)  
\psset{unit=0.5cm,linewidth=0.2pt,radius=0.08cm,labelsep=0.2cm}
\Cnode(0,0){v1}\nput{180}{v1}{1}
\Cnode(2,2){v2}\nput{135}{v2}{2}
\Cnode(2,4){v3}\nput{90}{v3}{3}
\Cnode(4,0){v4}\nput{0}{v4}{4}
\Cnode(2,-2){v5}\nput{270}{v5}{5}
\ncline{v1}{v2}
\ncline{v2}{v3}
\ncline{v2}{v4}
\ncline{v4}{v5}
\ncline{v5}{v1}
\ncline{v5}{v2}
\end{pspicture}
\begin{pspicture}(4,3)(-2,-1.5)  
\psset{unit=0.5cm,linewidth=0.2pt,radius=0.08cm,labelsep=0.2cm}
\Cnode(0,0){v11}
\Cnode(3,2){v21}
\Cnode(2.5,4){v31}
\Cnode(3.5,4){v32}
\Cnode(6,-1){v41}
\Cnode(6,0){v42}
\Cnode(6,1){v43}
\Cnode(2.5,-2){v51}
\Cnode(3.5,-2){v52}
\ncline{v11}{v21}
\foreach \x in {v31,v32} {\ncline{v21}{\x}}
\ncline{v31}{v32}
\foreach \x in {v41,v42,v43} {\ncline{v21}{\x} \foreach \y in {v51,v52} {\ncline{\x}{\y}}}
\ncline{v41}{v42}\ncline{v42}{v43}\ncarc[arcangle=-45]{v41}{v43}
\foreach \x in {v51,v52} {\ncline{\x}{v11} \ncline{\x}{v21}}
\ncline{v51}{v52}
\ncbox[nodesep=0.45,boxsize=0.6,linearc=0.5,linestyle=dashed]{v11}{v11}\nput[labelsep=0.6]{180}{v11}{1}
\ncbox[nodesep=0.45,boxsize=0.6,linearc=0.5,linestyle=dashed]{v21}{v21}\nput[labelsep=0.6]{135}{v21}{2}
\ncbox[nodesep=0.6,boxsize=0.6,linearc=0.5,linestyle=dashed]{v31}{v32}\nbput[npos=2.5]{3}
\ncbox[nodesep=0.6,boxsize=0.6,linearc=0.5,linestyle=dashed]{v41}{v43}\nbput[npos=0.5]{4}
\ncbox[nodesep=0.6,boxsize=0.6,linearc=0.5,linestyle=dashed]{v51}{v52}\nbput[npos=0.5]{5}
\end{pspicture}
\end{center}
\caption{Construction used in the proof of Theorem 1.1.}\label{fig1}
\end{figure*}

Consequently, if there are $m'_i$ non-edges of the form $h_jh_i$ in $H$, with $j< i$, then the number of vertices in $KH_i$ is one more: $i-m_i = m'_i+1$. Assume now that the vertices of $G$ are colored properly. We show now (inductively) that there exists a rainbow induced subgraph $H'$ of $G$ isomorphic to $H$. 
To this end from each clique $KH_i$ we choose one vertex $h'_i$. It is clear that the resulting induced subgraph is isomorphic to $H$. We need to demonstrate that the vertices can be chosen so that they have different colors. 

First, the only vertex of $KH_1$ belongs to $H'$, and suppose that from each $KH_j$, $j<i$, we have chosen a vertex $h'_j$ as one belonging to $H'$ so that all  $h'_1,...,h'_{i-1}$ are of different colors. We show that we can choose  $h'_i \in KH_i$ to keep this property. Indeed, if for some $1\leq j < i$, $(h_j,h_i)$ is an edge in $H$, then each vertex in  $KH_i$ has a different color than $h'_j$ (since all these vertices are neighbors of $h'_j$ in $G$). There are $m'_1$ vertices left among $h'_1,...,h'_{i-1}$, of different colors, to compare with the possible color of $h'_i$. Yet there are $m'_i+1$ vertices in $KH_i$, so we can choose one among them with the color different from all the colors chosen so far, as required.


In order to prove the first inequality we show that if $G$ is a required graph, such that in each coloring of $G$ there exists a rainbow induced subgraph $H$, it needs to have at least 
$\left\lceil \frac{n}{\chi} \right\rceil \left( n - \frac{\chi}{2}\left\lceil\frac{n}{\chi}-1\right\rceil\right)$ vertices. To this end we color $G$ step by step and count the number of colored vertices until $n$ colors is used. If $\chi = n$ the inequality is clear, so we may assume that $\chi < n$. 

First let us choose any induced subgraph $H_1$ of $G$ isomorphic to $H$, and color the vertices of $H_1$ into $\chi$ colors. This may be considered as a part of certain coloring of $G$. Hence it follows that there exist an induced subgraph $H_2$ of $G$ isomorphic to $H$ having in common with $H_1$ at most $\chi$ differently colored vertices. We color the remaining vertices of $H_2$ with (at most) $\chi$ new different colors. We proceed in such a way until $n$ color is used.  At the $i-th$ step of this procedure we have used not more than $i\cdot\chi$ colors. If this number is still less than $n$, than we choose a subgraph $H_{i+1}$ of $G$ isomorphic to $H$ having in common with the set of vertices colored so far at most $i\cdot\chi$ differently colored vertices. So, again, we color the remaining vertices of $H_{i+1}$ with at most $\chi$ new different colors. 

There are at least $k = \left\lceil \frac{n}{\chi} \right\rceil$ steps until $n$ colors are used.
The numbers of vertices colored at each step are at least, respectively, $n, n-\chi, n-2\chi,\ldots, n-(k-1)\chi$. It follows that the order of $G$ 
\begin{displaymath}
|G| \geq kn - \chi\frac{k(k-1)}{2} = k\left(n - \chi\frac{k-1}{2}\right),
\end{displaymath}
as required. 
\end{proof}

Note, that since the right hand side in the latter inequality, for $k < \frac{n}{\chi}+1$, is an increasing function of $k$, we have 
\begin{displaymath}
|G| \geq  \frac{n}{\chi}\left(n - \chi\frac{\frac{n}{\chi}-1}{2}\right) = \frac{n}{2}\left(\frac{n}{\chi}+1\right).
\end{displaymath}
This estimation shows clearly the order of the magnitude of the bound, but is generally weaker than the bound given in the theorem. It may be proven to be the same in cases when $n$ is a multiplicity of $\chi$. 

We note that for cliques $H=K_n$ and anticliques $H=A_n$ the bounds in the theorem are equal and thus give formulas $\rho(K_n)=n$ and $\rho(A_n)= {n(n+1)}/{2}$. 

For a double clique $H=2K_{n}$, the disjoint union of two cliques $K_{n}$, we have 
$3n\leq \rho(2K_{n}) \leq 2n+n^2$. In this case, it is not difficult to see, that the lower bound is attained. Indeed, the graph $G = K_n \cup K_{2n}$, the disjoint union of $K_n$ and $K_{2n}$, has a rainbow induced subgraph $2K_n$ in each coloring. 

An example, when only the upper bound gives the exact formula, is the star $S_n$ (that is the complete bipartite graph $K_{1,n-1}$). Our bounds yield 
\begin{displaymath}
\left\lceil n/2 \right\rceil \left\lfloor n/2+1 \right\rfloor\le \rho(S_n) \le n(n-1)/2+1.
\end{displaymath} 
One can prove directly (but it requires already some work) that 
\begin{equation}\label{star}
\rho(S_n) = n(n-1)/2+1.
\end{equation} 
It follows also from a more general result we prove in the next section.

Computing the exact value $\rho(H)$ is generally a hard problem. It involves computing both the clique number and the chromatic number. There is however one case when the value $\rho(H)$ has been already computed: these are the graphs $H$ for which the upper and lower bounds in Theorem~\ref{th1} coincide, thus giving the exact formula. We have observed already that this is a case when $H$ is a clique or anticlique. It turns out that these are special cases of a larger class, namely, the Tur\'an graphs $T(n,r)$, that is, the complete $r$-partite graph with $n$ vertices and as equal classes as possible (cf. \cite{KMSV}).

Indeed, for such a graph and $n=k\chi+s$, with $0<s\leq \chi$, we have $\left\lceil n/\chi \right\rceil = k+1$, and $\chi = r$. The left hand side of (\ref{eq1}) equals $(k+1)(s+ kr/2)$, while the left hand side equals $(r-s)k(k+1)/2 + s(k+1)(k+2)/2$, which is the same. Hence we have the formula
$\rho(T(n,r)) = \left\lceil n/\chi \right\rceil(n+s)/2$, where $s$ is the number of larger classes in $T(n,r)$. Which is more interesting, the equality of our bounds implies that $H$ must be a Tur\'an graph. In fact we have the following.

\begin{Theorem} The equality  of bounds (\ref{eq1}) in Theorem 1 holds if and only if $H$ is a Tur\'an graph. 
\end{Theorem}

\begin{proof}
The "if" part is proved above. It remains to prove the "only if" part. Assume that $T=T(n,\chi)$ is the Tur\'an graph with $\chi=\chi(H)$ classes. Then $\chi(T)=\chi$. From the Tur\'an's theorem \cite{Di} we obtain that $T$ has a maximal possible number of edges and it is unique for a fixed $n$ and $\chi$. By assumption and the properties of the Tur\'an graph established so far, we have
\begin{displaymath}m'(H) = \left\lceil \frac{n}{\chi} \right\rceil \left( n - \frac{\chi}{2}\left\lceil\frac{n}{\chi}-1\right\rceil\right) - n = m'(T).
\end{displaymath} 
It follows that the sizes of $T$ and $H$ are the same, and therefore by the maximality property mentioned above, $H=T$. 
\end{proof}

\section{The second lower bound}


We will consider possibilities of improving our general bounds for special classes of graphs. This can be done when we have additional information on the structure of the graph. First we 
make use of partitions of the set of vertices into anticliques. Note that for each graph $H$ there are usually many such partitions, if the anticliques of size one are admitted. For graphs with both large anticliques and large chromatic numbers our second lower bound below is much better.

\begin{Theorem} \label{th3}  If the set of vertices of a graph $H$ can be partitioned into anticliques of sizes $x_1, x_2, \ldots, x_k$, then 
\begin{equation}\label{eq3}
\rho(H) \geq \sum_{i=1}^k \frac{x_i(x_i + 1)}{2}.
\end{equation}\end{Theorem}

\begin{proof}
The idea of the proof is the same as before. Yet, now we apply a different more efficient coloring using the structure of anticliques. By assumption $H$ is a union of disjoint anticliques $A_{x_1}, A_{x_2}, \ldots, A_{x_k}$  and, possibly, some additional edges between anticliques. We may assume that the sizes satisfy $x_1 \geq x_2 \geq \ldots \geq x_k$. 

Let $G$ be a graph with $G\tor H$. As before, $G$ has to contain an induced copy $H_1$ of $H$. We color only the vertices in the largest anticlique of $H_1$ using one color and treat this as a part of a certain coloring of $G$. Since $G\tor H$, it follows that there exist an induced subgraph $H_2$ of $G$ isomorphic to $H$ having in common with the colored part at most one vertex $v$. With a new color we color now all the vertices in the largest anticlique of $G-v$. Again, it follows that there exist a subgraph $H_3$ of $G$ isomorphic to $H$ having in common with the colored part at most 2 vertices $u,w$, and we color with the third color all the vertices in the largest anticlique of $G-\{u,w\}$. We continue this process until we have used $n=|H|$ different colors. We claim that we have colored in such a way at least $\sum_{i=1}^k x_i(x_i + 1)/2$ vertices.

Indeed, suppose that there are $k_i \geq 0$ anticliques of size $i$ among $A_{x_1}, A_{x_2}, \ldots, A_{x_k}$, where $1 \leq i \leq m$, and $m=x_1$ is the size of the largest anticlique. In particular, $\sum_ik_i = n$. Then, in each of the first $k_m$ steps an uncolored anticlique of size $m=x_1$ appears, and hence the number of colored vertices is precisely $m$. In each of the next $k_m+k_{m-1}$ steps, an uncolored anticlique of size at least $m-1$ has to appear, and thus the number of colored vertices is at least $m-1$. Similarly, in each of the next $k_m+k_{m-1}+k_{m-2}$ steps the number of colored vertices is at least $m-2$, and so on. It follows that $G$ has at least $\sum_{j=1}^m j(k_m+k_{m-1}+\ldots+k_j)$ vertices. We have 
\begin{displaymath} 
\sum_{j=1}^m j(k_m+k_{m-1}+\ldots+k_j) = \sum_{j=1}^m k_j(1+2+\ldots+j).
\end{displaymath} 

The latter contains $k_m$ sums $(1+2+\ldots+m)$. Since, by definition of $k_m$, $m=x_1=x_2=\ldots=x_{k_m}$, 
these sums can be presented as
\begin{displaymath}(1+2+\ldots+x_1) + (1+2+\ldots+x_2) + \ldots + (1+2+\ldots+x_{k_m}).
\end{displaymath} 
Next, if $s<m$ is the largest index such that $k_s\neq 0$, then the next non-zero summands of the considered sum are $k_s$ sums 
\begin{displaymath}(1+2+\ldots+x_{k_m+1}) + (1+2+\ldots+x_{k_m+2}) + \ldots + (1+2+\ldots+x_{k_m+k_s}).
\end{displaymath}
Continuing in such a way we see that 
\begin{displaymath} 
\sum_{j=1}^m k_j(1+2+\ldots+j) = \sum_{i=1}^k (1+2+\ldots+x_i) = \sum_{i=1}^k \frac{x_i(x_i + 1)}{2},
\end{displaymath} 
as required.
\end{proof}

Let us observe that the bound above may be presented as $ n+\sum_{i=1}^k x_i(x_i - 1)/2.$ The latter term counts exactly the number of the non-edges in the anticliques. This suggests the following more transparent formulation of the result (combined with the upper bound in (\ref{eq1}))

\begin{Cor}Let the vertex set of a graph $H$ be partitioned into anticliques and let $m'_A$ denotes the number of the non-edges in these anticliques. If $m'$ is the number of all the non-edges in $H$,  and $n = |H|$, then the following holds
\begin{displaymath} n + m_A' \leq \rho(H) \leq n+m'\end{displaymath} 
\end{Cor}

Note that from this formulation it is immediate that both the bounds are equal if and only if $H$ is a complete $r$-partite graph. Hence we have

\begin{Cor} For the complete $r$-partite graph $\rho(K_{x_1,x_2,\ldots,x_k})$ we have  \begin{displaymath}\rho(K_{x_1,x_2,\ldots,x_k})=\sum_{i=1}^r \frac{x_i(x_i + 1)}{2}
\end{displaymath}\end{Cor}

In particular, for the star $S_n$ we have the result announced in (\ref{star}).
Also Tur\'an graphs appear in this corollary as a special case. In fact, Theorem~\ref{th3} improves generally the lower bound (\ref{eq1}), yet the argument is not straight and depends on the choosing a suitable anticlique partition.


\section{The second upper bound}

It is also possible to improve the upper bound of (\ref{eq1}) for some types of graphs using a special decomposition into cliques. Given a graph $H$ we define a relation $\sim$ on the set of vertices by $x\sim y$ if and only $xy\in H$ and $x$ and $y$ have exactly the same neighbors in $H-\{x,y\}$. We note that this is an equivalence relation and the equivalence classes form cliques in $H$. 
These cliques are called the \emph{replication cliques} of $H$. 

This name is justified by the fact that, if the relation $\sim$ is non-trivial, the graph may be viewed as one obtained from a smaller graph $M$ by \emph{replication} of vertices. More precisely, a graph $H$ is obtained from a graph $M$ by replication of a vertex $x\in M$ if $H$ is obtained from $M$ by replacing vertex $x$ by a clique $K$ of vertices, and replacing each edge $xy$ incident with $x$ by the edges joining $y$ with all vertices of $K$. A graph $H$ obtained from a graph $M$ by successive replication of vertices is called a \emph{replication graph} of $M$ (this definition is a little bit more general than that given in \cite{CZ}).

We note that the construction of the graph $G$ used in the proof of the upper bound in Theorem~\ref{th1} is the replication graph of $H$ consisting of the cliques $KH_1, KH_2,\ldots,KH_n$. The size of $KH_i$ is precisely $1+n_i$, where $n_i$ denotes the number of non-edges between the vertex of $h_{i}$ and the vertices  $\{h_{1}, h_{2}, \ldots, h_{i-1}\}$. We generalize this construction as follows.

\begin{Theorem} \label{th4}  If $KH_{1}, KH_{2} \ldots, KH_{s}$ are the replication cliques of a graph $H$ of sizes $y_1,y_2,\ldots y_s \geq 1$, respectively, then
\begin{equation}\label{eq4}
\rho(H) \leq n+ \sum_{i=1}^s n_i,
\end{equation}
where $n_i$ denotes the number of non-edges between a vertex of $KH_{i}$ and the vertices of $KH_{1} \cup KH_{2} \cup \ldots \cup KH_{i-1}$. The bound is the best when the sizes of the cliques satisfy $y_1 \leq y_2 \leq \ldots \leq y_s$. 
\end{Theorem}

\begin{proof}
To find a graph $G$ of the required size with $G\tor H$, we consider again a replication graph of $H$ with the cliques of suitable sizes. Each clique $KH_i$ is replicated (replaced) by a clique $KG_i$ of size $y_i+n_i$. (Note that since $n_1=0$ by definition, the cliques $KG_1$ and $KH_1$ are the same). 
Consider any coloring of the graph $G$ resulted in such a way. We describe by induction a rainbow induced subgraph $H'$ isomorphic to $H$. As the first part of $H'$ we take $KH_1=KG_1$. Since  $KH_1$ is a clique all the vertices in it have different colors. Suppose we have constructed a rainbow part $H'_{i-1}$ isomorphic to $KH_{1} \cup KH_{2} \cup \ldots \cup KH_{i-1}$ contained in $G'_{i-1}=KG_{1} \cup KG_{2} \cup \ldots \cup KG_{i-1}$. To obtain a part $H'_i$ isomorphic to $KH_{1} \cup KH_{2} \cup \ldots \cup KH_{i}$ it is enough to adjoin to $H'_{i-1}$ any subgraph of $KG_i$ of cardinality $y_i=|KH_i|$. We have only to make sure that the vertices have different colors from those in $H'_{i-1}$. 

By assumptions, there are $n_i$ non-edges between a vertex of $KG_i$ and the vertices of $H'_{i-1}$. Moreover, for each vertex in $KG_i$ the non-edges lead to the same vertices in $H'_{i-1}$. It follows that there are at most $n_i$ vertices in $KG_i$ they have a color of a vertex in $H'_{i-1}$. Consequently, there are at least $y_i$ vertices in $KG_i$ that have a color different from the vertices in $H'_{i-1}$. Hence, the clique they form is just as required.

In order to see that the bound is the best when the sizes of the cliques are in an increasing order, suppose we have two cliques $KG_{i-1}$ and $KG_i$ in $G$ of sizes  $y_{i-1}+n_{i-1}$ and  $y_i+n_i$, respectively, such that $y_{i-1} > y_i$. If we change the order of the cliques $KG_{i-1}$ and $KG_i$ in $G$ then the only part of the sum of (\ref{eq4}) that may change is that given by $n_{i-1}+n_i$. Let $n'_{i-1}$ and $n'_i$ denote the respective numbers of non-edges after the change of the order of the cliques $KG_{i-1}$ and $KG_i$. Now, there are two cases: (1) either there is no edge between $KG_{i-1}$ and $KG_i$, or (2) there is an edge between any two vertices of $KG_{i-1}$ and $KG_i$. In the first case, the change of the order of the cliques does not change the sum $n_{i-1}+n_i$. In the second case $n_i = m_i+ y_{i-1}$, where $m_i$ counts the non-edges between a vertex in $KH_i$ and the vertices in $KH_{1} \cup KH_{2} \cup \ldots \cup KH_{i-2}$. After the change of the order we have $n'_{i-1} = m_i$ and $n'_i = n_{i-1} + y_i$. Whence 
\begin{displaymath}n'_{i-1} +n'_i = n_{i-1}+m_i+ y_i < n_{i-1}+m_i+ y_{i-1}= n_{i-1}+n_i,
\end{displaymath}yielding a better bound. Now the claim easily follows.
\end{proof}

Combining the results of Theorems~\ref{th3} and \ref{th4} we obtain also the exact formula for the complements of complete $r$-partite graphs, that is, disjoint unions of cliques. 

\begin{Cor}
Let $H$ be the union of disjoint cliques $KH_{1} \cup KH_{2} \cup \dots \cup KH_{s}$ of sizes $y_1 \geq y_2 \geq \ldots \geq y_s$, respectively. Then 
\begin{displaymath}\rho(H) = \sum_{i=1}^s iy_i.
\end{displaymath}
\end{Cor}

\begin{figure*}
\begin{center}
\begin{pspicture}(7,1.7)(-2,-3.5)
$
\psset{unit=0.7cm,linewidth=0.2pt,radius=0.08cm,labelsep=0.5cm}
\rput(0,1.2){\rnode{k1}{KH_1}}
\rput(2,1.2){\rnode{k2}{KH_2}}
\rput(4,1.2){\rnode{k3}{\dots}}
\rput(6.2,1.2){\rnode{k4}{KH_{s-1}}}
\rput(8,1.2){\rnode{k5}{KH_s}}
              
\rput[l](-2.8,0){\rnode{a1}{AH_1}}\Cnode(0,0){v11}\Cnode(2,0){v21}\rput(4,0){\rnode{hdots1}{\dots}}\Cnode(6,0){v31}\Cnode(8,0){v41}
\rput[l](-2.8,-1){\rnode{a2}{AH_2}}\Cnode(0,-1){v12}\Cnode(2,-1){v22}\rput(4,-1){\rnode{hdots2}{\dots}}\Cnode(6,-1){v32}\Cnode(8,-1){v42}
\rput[c](-2.3,-2){\rnode{a3}{\vdots}}\rput(0,-2){\rnode{vdots1}{\vdots}}\rput(2,-2){\rnode{vdots2}{\vdots}}
\rput[l](-2.8,-3){\rnode{a4}{AH_{r-2}}}\Cnode(0,-3){v13}\Cnode(2,-3){v23}
\rput[l](-2.8,-4){\rnode{a5}{AH_{r-1}}}\Cnode(0,-4){v14}
\rput[l](-2.8,-5){\rnode{a6}{AH_r}}\Cnode(0,-5){v15}

\ncbox[nodesep=0.5,boxsize=0.5,linearc=0.5,linestyle=solid]{v11}{v15}
\ncbox[nodesep=0.5,boxsize=0.5,linearc=0.5,linestyle=solid]{v21}{v23}
\ncbox[nodesep=0.5,boxsize=0.5,linearc=0.5,linestyle=solid]{v21}{v23}
\ncbox[nodesep=0.5,boxsize=0.5,linearc=0.5,linestyle=solid]{v31}{v32}
\ncbox[nodesep=0.5,boxsize=0.5,linearc=0.5,linestyle=solid]{v41}{v42}
\ncbox[nodesep=0.8,boxsize=0.3,linearc=0.3,linestyle=dashed]{v11}{v41}
\ncbox[nodesep=0.8,boxsize=0.3,linearc=0.3,linestyle=dashed]{v12}{v42}
\ncbox[nodesep=0.8,boxsize=0.3,linearc=0.3,linestyle=dashed]{v13}{v23}
\ncbox[nodesep=0.8,boxsize=0.3,linearc=0.3,linestyle=dashed]{v14}{v14}
\ncbox[nodesep=0.8,boxsize=0.3,linearc=0.3,linestyle=dashed]{v15}{v15}
$
\end{pspicture}\end{center}
\caption{Partitioning the union of cliques into anticliques.}\label{fig3}
\end{figure*}
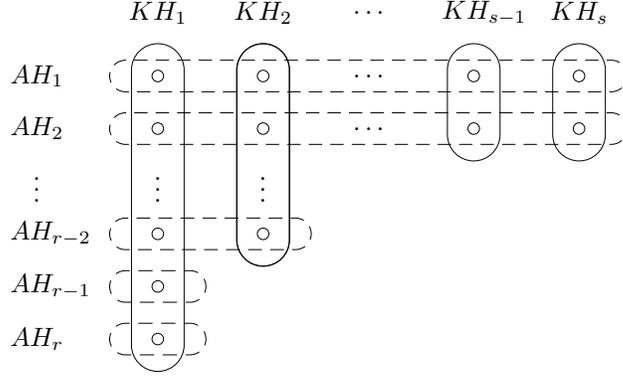

\begin{proof}
First we apply the lower bound of Theorem~\ref{th3}.  Note that $H$ can be partitioned naturally into $r=y_1$ anticliques $AH_{1} \cup AH_{2} \cup \dots \cup AH_{r}$ of the sizes $x_1 \geq x_2 \geq \ldots \geq x_r$, respectively (see Figure~\ref{fig3}). The size of the largest anticlique $AH_1$ is $x_1= s$, and there are exactly $y_s$ anticliques of this size. Then, the size of the next anticlique $AH_{y_s+1}$ is $s-1$ or less (in case when $y_s = y_{s-1}$). It follows that the right hand side of (\ref{eq3}) is
\begin{displaymath}
\sum_{i=1}^k \frac{x_i(x_i + 1)}{2} = \frac{(s+1)s}{2}y_s + \frac{s(s-1)}{2}(y_{s-1}-y_s) + \ldots + \frac{2\cdot 1}{2}(y_{1}-y_2) = 
\end{displaymath}
\begin{displaymath}
= \sum_{i=1}^s \frac{(i+1)i}{2}(y_i-y_{i+1}) =
\end{displaymath}
where $y_{s+1} =0$, and further calculation yields
\begin{displaymath}
=\sum_{i=1}^s \frac{(i+1)i}{2}(y_i-y_{i+1}) = \sum_{i=1}^s (\frac{(i+1)i}{2}-\frac{i(i-1)}{2})y_i =
\sum_{i=1}^s iy_i,
\end{displaymath}
as required.

Now we apply the upper bound of Theorem~\ref{th4}. According to the second statement of Theorem~\ref{th4} we consider our cliques in the reverse order. We note that $n_2 = y_s$, $n_3=y_s+y_{s-1}$,\ldots, $n_s = y_s+y_{s-1}+\ldots+y_2$, and $n = y_s+y_{s-1}+\ldots+y_1$. Summing everything up yields the desired result.
\end{proof}

\section{Paths and replication graphs}

While we are able to compute exact values $\rho(H)$ for some classes of graphs, we do not know such a value for the simplest kind of graphs---paths. Let $P_n$ denotes the path of length $n-1$.  Using (\ref{eq1}) we obtain for odd $n$: 
\begin{displaymath}(n+1)^2/4 \leq \rho(P_n) \leq 1+n(n-1)/2,
\end{displaymath} and for even $n$:   
\begin{displaymath}n(n+2)/4 \leq \rho(P_n) \leq 1+n(n-1)/2.
\end{displaymath}
(For paths, (\ref{eq3}) and (\ref{eq4}) yield no improvements.) This gives the exact value $\rho(P_3)=4$, and the following bounds for small $n$:  $6 \leq \rho(P_4) \leq 7$, $9 \leq \rho(P_5) \leq 11$, $12 \leq \rho(P_6) \leq 16, 16 \leq \rho(P_7) \leq 22$. Applying a suitable computer program generating all vertex colorings for small graphs, we have obtained the following: $\rho(P_4)=7, \rho(P_5)=10, \rho(P_6) \leq 14, \rho(P_7) \leq 19$. In particular, by computer search we have found the graph $G$ in Figure~\ref{fig2}, containing in every coloring a rainbow $P_5$.

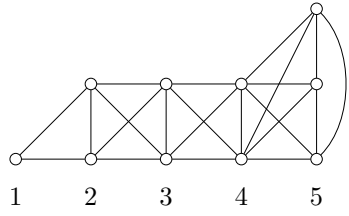
\begin{figure*}
\begin{center}
\begin{pspicture}(4,2.5)(0,-0.7)$ 
\psset{unit=1.0cm,linewidth=0.2pt,radius=0.08cm,labelsep=0.3cm}
\Cnode(0,0){p11}
\Cnode(1,0){p21}\Cnode(1,1){p22}
\Cnode(2,0){p31}\Cnode(2,1){p32}
\Cnode(3,0){p41}\Cnode(3,1){p42}
\Cnode(4,0){p51}\Cnode(4,1){p52}\Cnode(4,2){p53}
\nput{270}{p11}{1}
\nput{270}{p21}{2}
\nput{270}{p31}{3}
\nput{270}{p41}{4}
\nput{270}{p51}{5}
\ncline{p11}{p21}\ncline{p11}{p22}
\ncline{p21}{p22}
\ncline{p21}{p31}\ncline{p21}{p32}\ncline{p22}{p31}\ncline{p22}{p32}
\ncline{p31}{p32}
\ncline{p31}{p41}\ncline{p31}{p42}\ncline{p32}{p41}\ncline{p32}{p42}
\ncline{p41}{p42}
\ncline{p41}{p51}\ncline{p41}{p52}\ncline{p41}{p53}\ncline{p42}{p51}\ncline{p42}{p52}\ncline{p42}{p53}
\ncline{p51}{p52}\ncline{p52}{p53}\ncarc[arcangle=-45]{p51}{p53}
$\end{pspicture}
\caption{Graph $G=P_5(1,2,2,2,3)$.}\label{fig2}
\end{center}
\end{figure*}

This is a replication graph of $P_5$. We denote it $G=P_5(1,2,2,2,3)$ meaning that the vertices of $P_5$ are replaced by the cliques 
$K_1,K_2,K_2,K_2,K_3$ in the natural order. 
Our computer search shows that there are other replication graphs of $P_5$ of order $14$ containing rainbow $P_5$ in every coloring; for example, $G=P_5(3,1,1,2,3)\tor P_5$. In turn, 
the replication graph $G=P_6(2,2,3,3,2,2)$ satisfies $G\tor P_6$. A natural question arise whether 
$\rho(P_n)$ can be obtained always by means of replication graphs. More precisely: {is it true that for each $n\geq 2$ there exists a replication graph $G$ of $P_n$ of order $\rho(P_n)$ with $G\tor P_n$}? 

We observe that it is generally not the case for other graphs. In Figure~\ref{fig4} a graph $H$ is depicted and a graph $G$ beside with $G\tor H$, which shows that $\rho(H) \leq 7$. Yet, as it can be easily checked, no replication graph of $H$ of order $7$ or less has this property.

Let us define the number $\rho_R(H)$ to be the minimal order of a replication graph $G$ of $H$ such that in each coloring of $G$ there exists a rainbow induced copy of $H$ \emph{having exactly one vertex in each of the cliques} $K_i$ corresponding to a vertex $h_i$ in $H$. We will write $G\torr H$ in such a situation. This condition is stronger but fairly natural and leads to the following nice properties. 

\begin{Theorem}\label{cor}
For any graph $H$ the following conditions hold:
\begin{enumerate}
\item[\rm (i)] \begin{math}
\rho(H) \leq \rho_R(H).
\end{math}
\item[\rm (ii)] If  $H'$ is a subgraph of $H$ obtained from $H$ by removing some edges, then $\rho_R(H') \geq \rho_R(H).$
\item[\rm (iii)]  If $H$ is the union of two disjoint graphs $H_1$ and $H_2$, and possibly some additional edges between $H_1$ and $H_2$, and $m(H_1,H_2)$ is the number of non-edges between $H_1$ and $H_2$, then 
\begin{displaymath}\rho_R(H_1) + \rho_R(H_2) \leq \rho_R(H) \leq \rho_R(H_1) + \rho_R(H_2) + m(H_1,H_2).
\end{displaymath}
\end{enumerate}
\end{Theorem}

\begin{proof} 
The first inequality (i) is immediate from the definition. For (ii), suppose that $G'\torr H'$. Then $G$ with $G\torr H$ can be obtained from $G'$ just by adding suitable edges. Since adding edges creates no new proper coloring, $G$ has the required property and thus proves (ii). 

For (iii), let $G$ be a replication graph of order $\rho_R(G)$ with $G\torr H$. It contains disjoint induced replication graphs $G_1$ of $H_1$ and $G_2$ of $H_2$ with properties $G_1\torr H_1$ and $G_2\torr H_2$. This proves the first inequality. 

For the second, let $G_1$ and $G_2$ be replication graphs of $H_1$ and $H_2$ of orders $\rho_R(H_1)$ and $\rho_R(H_2)$, respectively,  with properties $G_1\torr H_1$ and $G_2\torr H_2$. We compose them into the replication graph of $H$ adding suitable edges between $H_1$ and $H_2$. To make sure that there are enough vertices of different colors (similarly as in the proof of Theorem~\ref{th1}) we enlarge replication cliques in $H_2$: each clique is enlarged with the number of vertices equal to the number of non-edges between a vertex of the clique and the vertices in $H_1$. This yields altogether $m(H_1,H_2)$ new vertices, proving the second inequality.
\end{proof}

\begin{figure*}
\begin{center}

\begin{pspicture}(3,1)(0,-1)
$
\psset{unit=1.0cm,linewidth=0.2pt,radius=0.08cm,labelsep=0.2cm}
\Cnode(0,0){v1}\Cnode(1,0){v2}\Cnode(2,0){v3}\Cnode(3,0){v4}
\ncline{v2}{v3}\ncline{v3}{v4}
\nput{270}{v1}{1}\nput{270}{v2}{2}\nput{270}{v3}{3}\nput{270}{v4}{4}
{\rput(1.6,1.2){\rnode{k1}{H}}}
$
\end{pspicture}
\begin{pspicture}(3,2)(-3,-1)
${\rput(0.4,1.2){\rnode{k1}{G}}}
\psset{unit=1.0cm,linewidth=0.2pt,radius=0.1cm,labelsep=0.2cm}
\Cnode(0,0){v1}\Cnode(1,0.14){v2}\Cnode(2,1){v3}\Cnode(3,0.14){v4}\Cnode(2.61,-0.84){v5}\Cnode(1.39,-0.84){v6}\Cnode(2,0){v7}
\ncline{v2}{v3}\ncline{v3}{v4}\ncline{v4}{v5}\ncline{v5}{v6}\ncline{v6}{v2}\ncline{v2}{v7}\ncline{v3}{v7}\ncline{v4}{v7}\ncline{v5}{v7}\ncline{v6}{v7}    
$
\end{pspicture}

\end{center}
\caption{An example of $H$ with $\rho(H) < \rho_R(H)$.}\label{fig4}
\end{figure*}
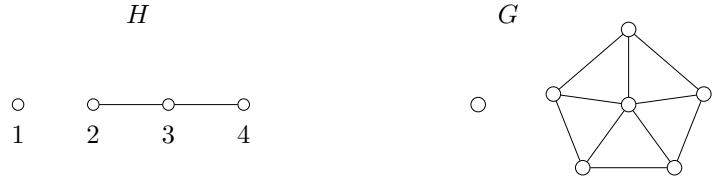

For graphs $H$ with non-trivial replication cliques it is possible to formulate a version of this result being the counterpart of Theorem~\ref{th4}. This requires however to redefine the number $\rho_R(H)$ suitably. We leave it to the reader. 

We observe that the results in (ii) and (iii) does not hold for $\rho(H)$. Indeed, graph $H'$ in Figure~\ref{fig5} is obtained from $H''$ by removing one edge, while the number $\rho(H') < \rho(H'')$ (the values of $\rho$ for graphs in Figure~\ref{fig5} have been obtained with a help of computer program). On the other hand, $H'$ is obtained from $H$ and vertex $u$, and some additional edges between $H$ and $u$, and the number $\rho(H)+\rho(K_1) = 10 > \rho(H')$ (disproving the first inequality in (iii) for $\rho$). Finally, let $H = P_4$, $H_1$ be $K_1$ consisting of a single vertex $u$, and $H_2$ be the disjoint union of $K_2$ and $K_1$. Then $P_4$ is the union of disjoint $H_1$ and $H_2$, and two additional edges joining $u$ with $H_2$, and we have
\begin{displaymath}
\rho(H_1) + \rho(H_2) + m(H_1,H_2) = 1 + 4 + 1 = 6 < 7= \rho(H),
\end{displaymath}
disproving the second inequality in (iii) for $\rho$.

\begin{figure*}\begin{center}
\begin{pspicture}(2,2.2)(0,0)
\large$
\psset{unit=1.0cm,linewidth=0.2pt,radius=0.08cm,labelsep=0.2cm}
\rput(1,1.7){\rnode{l}{H}}
\Cnode(0,0){v1}\Cnode(0,1){v2}\Cnode(1,1){v3}\Cnode(2,0){v4}\Cnode(2,1){v5}
\ncline{v2}{v3}\ncline{v3}{v4}\ncline{v3}{v5}\ncline{v4}{v5}
$
\end{pspicture}\hspace{2cm}
\begin{pspicture}(2,2)
\large$
\psset{unit=1.0cm,linewidth=0.2pt,radius=0.08cm,labelsep=0.2cm}
\rput(1,1.7){\rnode{l}{H'}}
\Cnode(0,0){v1}\Cnode(0,1){v2}\Cnode(1,0){v3}\Cnode(1,1){v4}\Cnode(2,0){v5}\Cnode(2,1){v6}
\ncline{v2}{v3}\ncline{v2}{v4}\ncline{v3}{v4}\ncline{v3}{v5}\ncline{v3}{v6}\ncline{v4}{v5}\ncline{v4}{v6}\ncline{v5}{v6}
\nput{270}{v3}{u}
$
\end{pspicture}\hspace{2cm}
\begin{pspicture}(2,1.7)
\large$
\psset{unit=1.0cm,linewidth=0.2pt,radius=0.08cm,labelsep=0.2cm}
\rput(1,1.7){\rnode{l}{H''}}
\Cnode(0,0){v1}\Cnode(0,1){v2}\Cnode(1,0){v3}\Cnode(1,1){v4}\Cnode(2,0){v5}\Cnode(2,1){v6}
\ncline{v2}{v3}\ncline{v2}{v4}\ncline{v3}{v4}\ncline{v3}{v5}\ncline{v3}{v6}\ncline{v4}{v5}\ncline{v4}{v6}\ncline{v5}{v6}
\ncline{v1}{v3}
\nput{270}{v3}{u}
$
\end{pspicture}\vspace*{4mm}\end{center}
\caption{Counterexamples: $\rho(H)=9$, $\rho(H')=9$, $\rho(H'')=10$.}\label{fig5}
\end{figure*}
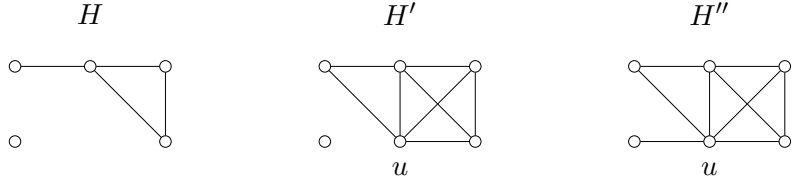

\section{Conclusions and open problems}

Obviously, the inequalities in (iii) of Theorem~\ref{cor} yields the best results, when the number $m(H_1,H_2)$ of non-edges is small. In particular, for \emph{join} of graphs we have the equality:

\begin{Cor}\label{corlast}
For any graphs $H_1$ and $H_2$ 
\begin{displaymath}
\rho_R(H_1 + H_2) = \rho_R(H_1) + \rho_R(H_2).
\end{displaymath}
\end{Cor}

We do not have any counterexample in case of $\rho$, and we conjecture that this property may be true for $\rho$, as well. Nevertheless, $\rho_R$ seems to be more convenient tool for computations, and therefore any result showing that the difference between $\rho$ and $\rho_R(H)$ cannot be large would be desirable. 

Unfortunately, our results cannot be very helpful for paths, since paths have many non-edges. Nevertheless, combining (iii) of Theorem~\ref{cor} with the inequality $\rho_R(P_7) \leq 19$ (obtained by computer search) we can infer inductively the following improvement of the general upper bound for paths:

\begin{Cor}
\begin{displaymath}\rho(P_n) \leq 1+n(n-1)/2 - 3\lfloor n/7 \rfloor.
\end{displaymath}
\end{Cor}

This still can be improved by $2$ or $1$, when the remainder of $n$ modulo $7$ is greater than $5$ or $4$, respectively (using other inequalities obtained by computer search mentioned above).

The question what is the exact value $\rho(P_n)$ remains open. Perhaps it can be approached by considering related problems for $\rho_R(P_n)$: 1) what is the exact formula for $\rho_R(P_n)$, and 2) is it true that  $\rho(P_n) = \rho_R(P_n)$ for all $n$?
These questions do not seem easy, but we think they are very interesting from combinatorial point of view.


\begin{thebibliography}{1}





\bibitem{AC} Axenovich, M.,  Choi, J.
\emph{On colorings avoiding  a rainbow cycle and a fixed monochromatic subgraph},
Electron. J. Combinatorics, 17(1) (2010), 12 pp.


\bibitem{AI} Axenovich, M., Iverson, P. 
\emph{Edge-colorings avoiding rainbow and monochromatic subgraphs},
Discrete Math.,  308 (20), (2008),  4710-4723.


\bibitem{AM} M. Axenovich, R. Martin, \emph{Avoiding rainbow induced subgraphs in vertex-colorings}. { Electron. J. Combinatorics 15(1), (2008), 23 pp.}

\bibitem{AS} Axenovich, M., Sackett, C.,
\emph{Avoiding rainbow induced subgraphs in edge-colorings},
Australaisian J. Combinatorics, 44, (2009), 287-296.


\bibitem{BCP} H. Broersma, A. Capponi, G. Paulusma, \emph{A new algorithm for on-line coloring bipartite graphs}. {SIAM J. Discrete Math. 22(1), (2008),  72--91.}

\bibitem{Di} R. Diestel, {"Graph Theory"}, {Springer-Verlag Heidelberg, New York 2005.}

\bibitem{CZ} G. Chartrand, P. Zhang, {"Chromatic Graph Theory"}, {Discrete Mathematics and Its Applications, Boca Raton 2009.}


\bibitem{Ma} G. Matecki G., \emph{On-line graph coloring on a bounded board}, PhD Thesis,{Computer Science Department, Faculty of Mathematics and Computer Science, Jagiellonian University (2006).}  \url{http://tcs.uj.edu.pl/docs.php?id=8}



\bibitem{KMSV} 
P. Keevash, D. Mubayi, B. Sudakov, J. Verstraëte, \emph{Rainbow Tur\'an Problems}, Combinatorics, Probability, and Computing, 16 (1) (2007). pp. 109-126.


\bibitem{KT} H. A. Kierstead, W. T. Trotter, 
\emph{Colorful induced subgraphs}, Discrete Math. 101 (1992), no. 1-3, 165–169.



\bibitem{KL} Y. Kohayakawa, T. Luczak \emph{Sparse Anti-Ramsey Graphs}, 
J. Combinatorial Theory, Series B 63(1995),  146-152.


\end{thebibliography}
\end{document}